\documentclass[a4paper,12pt]{article}

\newenvironment{proof}[1][]{\noindent {\bf Proof #1:\;}}{\hfill $\Box$}

\newtheorem{theorem}{Theorem}
\newtheorem{lemma}{Lemma}

\newtheorem{assumption}{Assumption}

\usepackage{amsmath}
\usepackage{amssymb}
\usepackage{amsfonts}
\usepackage{graphicx}
\usepackage{mathrsfs}

\title{\bf Linear conic optimization for nonlinear optimal control}

\begin{document}

\author{Didier Henrion$^{1,2,3}$, Edouard Pauwels$^{1,2}$}

\footnotetext[1]{CNRS, LAAS, 7 avenue du colonel Roche, F-31400 Toulouse, France.}
\footnotetext[2]{Universit\'e de Toulouse, LAAS, F-31400 Toulouse, France.}
\footnotetext[3]{Faculty of Electrical Engineering, Czech Technical University in Prague,
Technick\'a 2, CZ-16626 Prague, Czech Republic}

\date{Draft of \today}

\maketitle

\begin{abstract}
Infinite-dimensional linear conic formulations are described for nonlinear optimal control problems. The primal linear problem
consists of finding occupation measures supported on optimal relaxed controlled trajectories, whereas the
dual linear problem consists of finding the largest lower bound on the value function of the optimal control problem.
Various approximation results relating the original optimal control problem and its linear conic formulations
are developed. As illustrated by a couple of simple examples, these results are relevant in the context of finite-dimensional
semidefinite programming relaxations used to approximate numerically the solutions of
the infinite-dimensional linear conic problems.
\end{abstract}

\section{Motivation}

In \cite{l00,l01}, J.-B. Lasserre described a hierarchy of convex semidefinite programming (SDP) problems
allowing to compute bounds and find global solutions for finite-dimensional
nonconvex polynomial optimization problems. Each step in the hierarchy consists of solving 
a primal moment SDP problem and a dual polynomial sum-of-squares (SOS) SDP problem
corresponding to discretizations of infinite-dimensional linear conic problems, namely
a primal linear programming (LP) problem on the cone of nonnegative measures, and
a dual LP problem on the cone of nonnegative continuous functions.
The number of variables (number of moments in the primal SDP, degree of the SOS certificates
in the dual SDP) increases when progressing in the hierarchy, global optimality
can be ensured by checking rank conditions on the moment matrices, and global optimizers
can be extracted by numerical linear algebra. For more information on
the moment-SOS hierarchy and its applications, see \cite{l10}.

This approach was then extended to polynomial optimal control in \cite{lhpt08}.
Whereas the key idea in \cite{l00,l01} was to reformulate a (finite-dimensional) nonconvex polynomial optimization
on a compact semi-algebraic set into an LP in the (infinite-dimensional) space
of probability measures supported on this set, the key idea in \cite{lhpt08}, also developed
in \cite{gq09}, was to reformulate an (infinite-dimensional) nonconvex polynomial optimal control problem
with compact constraint set into an LP in the (infinite-dimensional) space
of occupation measures supported on this set.
Note that LP formulations of optimal control problems (on
ordinary differential equations and partial differential equations) are classical, and can be
traced back to the work by L. C. Young, Filippov, as well as Warga and Gamkrelidze,
amongst many others. For more details and a historical survey, see e.g. \cite[Part III]{f99}.
We believe that what is innovative in \cite{lhpt08} is the observation that the infinite-dimensional
linear formulations for optimal control problems can be solved numerically with a moment-SOS hierarchy of the
same kind as those used in polynomial optimization.

The objective of this contribution is to revisit the approach of \cite{lhpt08}
and to survey the use of occupation measures to
linearize polynomial optimal control problems. This is an opportunity to describe
duality in infinite-dimensional conic problems, as well as various approximation
results on the value function of the optimal control problem. The primal LP
consists of finding occupation measures supported on optimal relaxed controlled trajectories,
whereas the dual LP consists of finding the largest lower bound on the value function of the optimal control
problem. The value function is the solution (in a suitably defined weak sense) of a nonlinear partial
differential equation called the Hamilton-Jacobi-Bellman equation, see e.g. \cite[Chapters 8 and 9]{t05}
and \cite[Chapters 19 and 24]{c13}. It is traditionally used for verification of optimality, and for explicit
computation of optimal control laws, but we do not describe these applications here.

\section{Polynomial optimal control}

We consider polynomial optimal control problems (POCPs) of the form
\begin{equation}\label{pocp}
\begin{array}{llll}
v^*(t_0,x_0) & := & \inf_u & \displaystyle \int_{t_0}^T l(x(t),u(t))dt + l_T(x(T)) \\
&&  \mathrm{s.t.} & \dot{x}(t) = f(x(t),u(t)), \:\:x(t_0) = x_0 \\ 
&&& x(t) \in X,\: t \in [t_0,T] \\
&&& u(t) \in U, \: t \in [t_0,T] \\
&&& x(T) \in X_T
\end{array}
\end{equation}
where the dot denotes time derivative, 
$l \in {\mathbb R}[x,u]$ is a given Lagrangian (integral cost),
$l_T \in {\mathbb R}[x]$ is a given terminal cost,
$f \in {\mathbb R}[x,u]^n$ is a given dynamics (vector field),
$X \subset {\mathbb R}^n$ is a given compact state constraint set,
$U \subset {\mathbb R}^m$ is a given compact control constraint set,
$X_T \subset X$ is a given compact terminal state constraint set. Also given
are the terminal time  $T\geq 0$, the initial time $t_0 \in [0,T]$ and the initial condition
$x_0 \in X$. In POCP (\ref{pocp}), the minimum is with respect to all control laws
$u \in {\mathscr L}^{\infty}([t_0,T]; U)$
which are bounded functions of time $t$ with values in $U$, and the
resulting state trajectories $x \in {\mathscr L}^{\infty}([t_0,T]; X)$ which are
bounded functions of time $t$ with values in $X$.

Let ${\mathscr A} \subset [0,T]\times X$ denote the set of values $(t_0,x_0)$
for which there is a controlled trajectory $(x,u) \in {\mathscr L}^{\infty}([t_0,T]; X\times U)$
starting at $x(t_0)=x_0$ and admissible for POCP (\ref{pocp}).
The function $(t_0,x_0) \mapsto v^*(t_0,x_0)$ defined in (\ref{pocp})
is called the value function,  and its domain is ${\mathscr A}$.

\section{LP formulation}\label{sec:occupation}

As explained in the introduction, to derive an LP formulation of POCP (\ref{pocp}) we have to introduce
measures on trajectories, the so-called occupation measures.
The first step is to replace classical controls with probability measures,
and for this we have to define additional notations.

Given a compact set $X \subset {\mathbb R}^n$, let ${\mathscr C}(X)$ denote the
space of continuous functions supported on $X$, and let ${\mathscr C}_+(X)$
denote its nonnegative elements, the cone of nonnegative continuous
functions on $X$. Let ${\mathscr M}_+(X) = {\mathscr C}_+(X)'$
denote its topological dual, the set of all nonnegative continuous linear functional on ${\mathscr C}(X)$.
By a Riesz representation theorem, these are nonnegative Borel-regular measures,
or Borel measures, supported on $X$. The topology in ${\mathscr C}_+(X)$ is the strong topology
of uniform convergence, whereas the topology in ${\mathscr M}_+(X)$ is the weak-star topology.
The duality bracket 
\[
\langle v,\mu \rangle := \int_X v(x)\mu(dx)
\]
denotes the integration of a function $v \in {\mathscr C}_+(X)$ against a
measure $\mu \in {\mathscr M}_+(X)$. For background on weak-star
topology see e.g. \cite[Section 5.10]{l69} or \cite[Chapter IV]{b02}.
Finally, let us denote by ${\mathscr P}(X)$ the set of probability measures supported on $X$,
consisting of Borel measures $\mu \in {\mathscr M}_+(X)$ such that $\langle 1,\mu \rangle =1$.

\subsection{Relaxed controls}

In POCP (\ref{pocp}), given $(t_0,x_0) \in {\mathscr A}$,
let $(x_k,u_k)_{k\in\mathbb N} \in {\mathscr L}^{\infty}([t_0,T]; X\times U)$
denote a minimizing sequence of admissible controlled trajectories, i.e. it holds
\[
\dot{x_k}(t) = f(x_k(t),u_k(t)), \quad x(t_0) = x_0
\]
and
\[
\lim_{k\to \infty} \int_{t_0}^T l(x_k(t),u_k(t))dt + l_T(x_k(T)) = v^*(t_0,x_0).
\]
In general the infimum in POCP (\ref{pocp}) is not attained, so our next step is to assume that,
at each time $t \in [t_0,T]$, the control
is not a vector $u(t) \in U$, but a time-dependent probability
measure $\omega(du \:|\: t) \in {\mathscr P}(U)$ which rules the distribution
of the control in $U$. We use the notation $\omega_t := \omega(. \:|\:t)$
to emphasize the dependence on time. This is called a relaxed control, or stochastic control,
or Young measure in the functional analysis literature. POCP (\ref{pocp})
is then relaxed to
\begin{equation}\label{relaxedpocp}
\begin{array}{llll}
v^*_R(t_0,x_0) & := & \min_\omega & \displaystyle \int_{t_0}^T \langle l(x(t),.), \omega_t \rangle dt + l_T(x(T)) \\
&&  \mathrm{s.t.} & \displaystyle \dot{x}(t) = \langle f(x(t),.), \omega_t \rangle, \:\:x(t_0) = x_0 \\ 
&&& x(t) \in X, \: t \in [t_0,T] \\
&&& \omega_t \in {\mathscr P}(U), \: t \in [t_0,T]  \\
&&& x(T) \in X_T
\end{array}
\end{equation}
where the minimization is w.r.t. a relaxed control. Note that we
replaced the infimum in POCP (\ref{pocp}) with a minimum in relaxed
POCP (\ref{relaxedpocp}). Indeed, it can be proved that this minimum is always
attained using (weak-star) compactness of the space of probability measures with
compact support. 

Since classical controls $u \in {\mathscr L}^{\infty}([t_0,T]; U)$ 
are a particular case of relaxed controls $\omega_t \in {\mathscr P}(U)$ corresponding to the choice
$\omega_t = \delta_{u(t)}$ for a.e. $t \in [t_0,T]$,
the minimum in relaxed POCP (\ref{relaxedpocp}) is smaller than the infimum in classical POCP (\ref{pocp}), i.e.
\[
v^*(t_0,x_0) \geq v^*_R(t_0,x_0).
\]
Contrived optimal control problems (e.g. with overly stringent state constraints)
can be cooked up such that the inequality is strict, i.e. $v^*(t_0,x_0) > v^*_R(t_0,x_0)$,
see e.g. the examples in \cite[Appendix C]{hk14}.
We do not consider that these examples are practically relevant, and hence the following assumption
will be made.

\begin{assumption}[No relaxation gap]\label{norelaxationgap}
For any relaxed controlled trajectory $(x, \omega_t)$ admissible for relaxed POCP (\ref{relaxedpocp}),
there is a sequence of controlled trajectories $(x_k, u_k)_{k \in \mathbb{N}}$ admissible for POCP (\ref{pocp})
such that
\[
\lim_{k\to\infty} \int_{t_0}^T v(x_k(t), u_k(t)) dt = \int_{t_0}^T \langle v(x(t),.), \omega_t \rangle dt
\]
for every function $v \in {\mathscr C}(X\times U)$. Then it holds
\[
v^*(t_0,x_0)=v^*_R(t_0,x_0)
\]
for every $(t_0,x_0) \in {\mathscr A}$. 
\end{assumption}

Note that this assumption is satisfied under the classical controllability and/or
convexity conditions used in the Filippov-Wa$\dot{\text z}$ewski Theorem with
state constraints, see \cite{fr00} and the discussions around Assumption I in \cite{gq09}
and Assumption 2 in \cite{hk14}.
However, let us point out that Assumption \ref{norelaxationgap} does not imply that the infimum is
attained in POCP (\ref{pocp}).  Conversely, if the infimum is attained,
the values of POCP (\ref{pocp}) and relaxed POCP (\ref{relaxedpocp}) coincide,
and Assumption \ref{norelaxationgap} is satisfied.

\subsection{Occupation measure}

Given initial data $(t_0,x_0) \in {\mathscr A}$, and  given a relaxed control $\omega_t \in {\mathscr P}(U)$,
the unique solution of the ODE
\begin{equation}\label{rode}
\dot{x}(t) = \langle f(x(t),.), \omega_t \rangle, \:\: x(t_0) = x_0
\end{equation}
 in relaxed POCP (\ref{relaxedpocp}) is given by
\begin{equation}\label{state}
x(t) = x_0 + \int_{t_0}^t \langle f(x(s),.), \omega_s \rangle ds
\end{equation}
for every $t \in [t_0,T]$. Let us then define
\begin{equation}\label{occupation}
\mu(dt, dx, du) := dt \: \delta_{x(t)}(dx) \: \omega_t(du)  \in {\mathscr M}_+([t_0,T]\times X\times U)
\end{equation}
as the occupation measure concentrated uniformly in time on the state trajectory starting at $x_0$ at time $t_0$, for the given relaxed control $\omega_t$. An analytic intepretation is that integration w.r.t. the occupation measure
is equivalent to time-integration along system trajectories, i.e.
\[
\int_{t_0}^T v(t,x(t)) dt = \int_{t_0}^T \int_X \int_U v(t,x)\mu(dt,dx,du) = \langle v,\mu  \rangle
\]
given any test function $v \in {\mathscr C}([t_0,T]\times X)$.

Let us define the linear operator ${\mathcal L} : {\mathscr C}^1([t_0,T]\times X) \to {\mathscr C}([t_0,T]\times X\times U)$ by
\[
v \mapsto {\mathcal L} v := \frac{\partial v}{\partial t} + \sum_{i=1}^n \frac{\partial v}{\partial x_i} f_i
= \frac{\partial v}{\partial t} + \mathrm{grad}\:v \cdot f.
\]
Given a continuously differentiable test function $v \in {\mathscr C}^1([t_0,T]\times X)$,
notice that
\[
\begin{array}{rclcl}
v(T,x(T)) - v(t_0,x(t_0)) & = & \int_{t_0}^T dv(t,x(t)) & = & \int_{t_0}^T \dot{v}(t,x(t))dt \\[.5em]
& = & \int_{t_0}^T {\mathcal L} v(t,x(t)) dt & = & \langle {\mathcal L} v, \mu \rangle
\end{array}
\]
which can be written more concisely as
\begin{equation}\label{liouvilledual}
 \langle v,\mu_T \rangle - \langle v,\mu_0 \rangle = \langle {\mathcal L} v, \mu \rangle
\end{equation}
upon defining respectively the initial and terminal occupation measures
\begin{equation}\label{initialterminal}
\mu_0(dt,dx) := \delta_{t_0}(dt)\delta_{x(t_0)}(dx), \quad \mu_T(dt,dx):=\delta_T(dt)\delta_{x(T)}(dx).
\end{equation}

Let us define the adjoint linear operator ${\mathcal L}' : {\mathscr C}([t_0,T]\times X)'
\to {\mathscr C}^1([t_0,T]\times X\times U)'$ by
the relation $\langle v,{\mathcal L}'\mu \rangle := \langle {\mathcal L}v,\mu \rangle$
for all $\mu \in {\mathscr M}([t_0,T]\times X)$ and $v \in {\mathscr C}^1([t_0,T]\times X)$.
More explicitly, this operator can be expressed as
\[
\mu \mapsto {\mathcal L}' \mu = -\frac{\partial \mu}{\partial t}-\sum_{i=1}^n \frac{\partial (f_i\mu)}{\partial x_i}
= -\frac{\partial \mu}{\partial t} - \mathrm{div}\:f\mu
\]
where the derivatives of measures are understood in the weak sense, i.e. via their action
on smooth test functions, and the change of sign comes from  integration by parts.
Equation (\ref{liouvilledual}) can be rewritten equivalently as
$\langle v,\mu_T \rangle - \langle v,\mu_0 \rangle = \langle v,{\mathcal L}' \mu \rangle$
and since this equation should hold for all test functions $v \in {\mathscr C}^1([t_0,T]\times X)$, we obtain
a linear partial differential equation (PDE) on measures ${\mathcal L}' \mu = \mu_T - \mu_0$
that we write
\begin{equation}\label{liouville}
\frac{\partial \mu}{\partial t}+\mathrm{div}\:f \mu + \mu_T = \mu_0.
\end{equation}
This linear transport equation is classical in fluid mechanics, statistical physics
and analysis of PDEs. It is called the
equation of conservation of mass, or the continuity equation, or the advection equation,
or Liouville's equation.  Under the assumption that the initial  data $(t_0,x_0) \in {\mathscr A}$ and
the control law $\omega_t \in {\mathscr P}(U)$ are given, the following result
can be found e.g. in  \cite[Theorem 5.34]{v03} or \cite{a08}.

\begin{lemma}[Liouville PDE = Cauchy ODE]
There exists a unique solution to the Liouville PDE (\ref{liouville})
which is concentrated on the solution of the Cauchy ODE (\ref{rode}), i.e.
such that (\ref{occupation}) and (\ref{initialterminal}) hold.
\end{lemma}

In our context of conic optimization, the relevance of the Liouville PDE
(\ref{liouville}) is its linearity in the occupation measures $\mu$, $\mu_0$ and $\mu_T$, 
whereas the Cauchy ODE (\ref{rode}) is nonlinear in the state trajectory $x(t)$.

\subsection{Primal LP on measures}

The cost in relaxed POCP (\ref{relaxedpocp}) can therefore be written
\[
\int_{t_0}^T \langle l(x(t),.),\omega_t \rangle dt + l_T(x(T)) = \langle l,\mu \rangle + \langle l_T,\mu_T \rangle
\]
and we can now define a relaxed optimal control problem
as an LP in the cone of non-negative measures:
\begin{equation}\label{measpocp}
\begin{array}{llll}
	p^*(t_0,x_0) & := & \min_{\mu, \mu_T} &\langle l,\mu \rangle + \langle l_T,\mu_T \rangle \\
&&  \mathrm{s.t.} & \frac{\partial \mu}{\partial t}+\mathrm{div}\:f \mu + \mu_T = \delta_{t_0}\delta_{x_0} \\ 
&&& \mu \in {\mathscr M}_+([t_0,T]\times X\times U)\\
&&&  \mu_T \in {\mathscr M}_+(\{T\}\times X_T)
\end{array}
\end{equation}
where the minimization is w.r.t. the occupation measure $\mu$ (which includes the relaxed control $\omega_t$, see (\ref{occupation})) and the terminal measure $\mu_T$, for a given initial measure $\mu_0 = \delta_{t_0}\delta_{x_0}$
which is the right-hand side in the Liouville equation constraint.

Note that in LP (\ref{measpocp}) the infimum is always attained since the admissible set is (weak-star) compact
and the functional is linear. However, since classical trajectories are a particular case of relaxed trajectories corresponding
to the choice (\ref{occupation}), the minimum in LP (\ref{measpocp}) is smaller than the minimum
in relaxed POCP (\ref{relaxedpocp}) (this latter one being equal to the infimum in POCP (\ref{pocp}), recall Assumption
\ref{norelaxationgap}), i.e.
\begin{equation}\label{vpbound}
v^*(t_0,x_0) \geq p^*(t_0,x_0). 
\end{equation}
The following result, due to \cite{v93}, essentially based on convex duality,
shows that there is no gap occuring when considering more general occupation measures
than those concentrated on solutions of the ODE.

\begin{lemma}\label{lemmaVinter}
It holds $v^*(t_0,x_0)=p^*(t_0,x_0)$ for all $(t_0,x_0) \in {\mathscr A}$.
\end{lemma}

\subsection{Dual LP on functions}

Primal measure LP (\ref{measpocp}) has a dual LP in the cone of nonnegative continuous functions:
\begin{equation}\label{contpocp}
\begin{array}{llll}
d^*(t_0,x_0) & := & \sup_v & v(t_0,x_0) \\
&&  \mathrm{s.t.} & l+\frac{\partial v}{\partial t}+\mathrm{grad}\:v\cdot f \in {\mathscr  C}_+([t_0,T]\times X\times U) \\
&&& l_T-v(T,.) \in {\mathscr C}_+(X_T)
\end{array}
\end{equation}
where maximization is with respect to a continuously differentiable function $v \in {\mathscr C}^1([t_0,T]\times X)$
which can be interpreted as a Lagrange multiplier of the Liouville equation in (\ref{measpocp}).

In general the supremum in dual LP (\ref{contpocp}) is not attained, and weak duality
with  the primal LP (\ref{measpocp}) holds
\[
p^*(t_0,x_0) \geq d^*(t_0,x_0)
\]
but it can be shown that there is actually no duality gap:

\begin{lemma}[No duality gap]\label{nodualitygap}
It holds $p^*(t_0,x_0)=d^*(t_0,x_0)$ for all $(t_0,x_0) \in {\mathscr A}$.
\end{lemma}

\begin{proof}
The proof follows along the same lines as the proof of \cite[Theorem 2]{hk14}.
First we observe that $(t_0,x_0) \in {\mathscr A}$ and Assumption \ref{norelaxationgap}
imply that $p^*(t_0,x_0)$ is finite. Second, we use the condition that the cone
\[
\{(\langle l,\mu \rangle + \langle l_T,\mu_T \rangle, \:  \frac{\partial \mu}{\partial t}+\mathrm{div}\:f \mu + \mu_T) \: :\:
\mu \in {\mathscr M}_+([t_0,T]\times X\times U), \: \mu_T \in {\mathscr M}_+(\{T\}\times X_T)\}
\]
is closed in the weak-star topology.  This is a classical sufficient condition for
the absence of a duality gap in infinite-dimensional LPs, see e.g.
\cite[Chapter IV, Theorem 7.2]{b02}.
\end{proof}

\section{Approximation results}\label{sec:approx}

Primal LP (\ref{measpocp}) and dual LP (\ref{contpocp}) are infinite-dimensional
conic problems. If we want to solve them with a computer, we invariably
have to use discretization and approximation schemes. The aim of this section is
to derive various approximation results that prove useful when designing
numerical methods based on moment-SOS hierarchies.

\subsection{Lower bound on value function}

First, observe that there always exists an admissible solution for dual LP (\ref{contpocp}).
For example, choose $v(t,x) := a+b(T-t)$ with $a \in \mathbb R$ such that $l_T(x) \geq a$ on $X_T$
and $b \in \mathbb R$ such that $l(x,u) \geq b$ on $X\times U$. Moreover,
by construction, any admissible function for dual LP (\ref{contpocp})
gives a global lower bound on the value function:

\begin{lemma}[Lower bound on value function]\label{lowerbound}
If  $v \in {\mathscr C}^1([t_0,T]\times X)$ is admissible for dual LP (\ref{contpocp}),
then $v^* \geq v$  on $[t_0,T]\times X$.
\end{lemma}

\begin{proof}
If $(t_1, x_1) \notin {\mathscr A}$, then $v^*(t_1, x_1)$ is unbounded and the statement holds because $v(t_1, x_1)$ must be finite.
Let $(t_1,x_1)  \in {\mathscr A}$ be given. If $v$ is admissible for dual LP (\ref{contpocp}), then
for any admissible trajectory $(x,u) \in {\mathscr L}^{\infty}([t_0,T]; X\times U)$,
starting at $x(t_1)=x_1$, it holds
$\int_{t_1}^T (l(x(t),u(t)) + {\mathcal L}v(t,x(t),u(t)))dt =
\int_{t_1}^T l(x(t),u(t))dt + v(T,x(T))-v(t_1,x_1) \geq 0$
since $l+\mathcal L\geq 0$ on $[t_0,T]\times X\times U$.
Moreover $\int_{t_1}^T l(x(t),u(t))dt + l_T(x(T)) \geq \int_{t_1}^T l(x(t),u(t))dt + v(T,x(T))$
since $l_T-v(T,.) \geq 0$ on $X_T$. Combining the two inequalities yields
$\int_{t_1}^T l(x(t),u(t))dt + l_T(x(T)) \geq v(t_1,x_1)$
and the expected inequality follows by taking the infimum over admissible trajectories.
\end{proof}

The relation between the dual LP (\ref{contpocp}) and the original POCP (\ref{pocp}) is given by the following result:
\begin{lemma}[Maximizing sequence]
	\label{pointwiseConvergence}
Given $(t_0,x_0) \in {\mathscr A}$, there is a sequence $(v_k)_{k\in \mathbb N}$ admissible for the dual LP
(\ref{contpocp}) such that $v^*(t_0,x_0) \geq v_k(t_0,x_0)$ and $\lim_{k\to\infty} v_k(t_0,x_0) = v^*(t_0,x_0)$.
\end{lemma}

\begin{proof}
From Lemma \ref{lowerbound}, it holds $v^*(t_0,x_0) \geq v_k(t_0,x_0)$ for every function
$v_k$ admissible for dual LP (\ref{contpocp}). By Assumption \ref{norelaxationgap} and Lemma \ref{nodualitygap},
it holds $d^*(t_0,x_0)=p^*(t_0,x_0)=v^*(t_0,x_0)$ and hence there exists a maximizing sequence
$v_k \in {\mathscr C}^1([t_0,T]\times X)$ for LP (\ref{contpocp}).
\end{proof}

\subsection{Uniform approximation of the value function along trajectories}

In this section, we investigate the properties of maximizing sequences given by Lemma \ref{pointwiseConvergence},
and in particular their convergence to the value function of POCP (\ref{pocp}).
We first demonstrate the lower semicontinuity of the value of LP (\ref{measpocp}).
This leads to the lower semicontinuity of the value of
POCP (\ref{pocp}), by considering Assumption \ref{norelaxationgap} and Lemma \ref{lemmaVinter}.
Note that lower semicontinuity is readily ensured when the set
$\{(f(x,u), l(x,u)+a) \: :\: u \in U, a \geq 0\}$
is convex in ${\mathbb R}^{n+1}$ for all $x$, with $U$ compact, see e.g. \cite[Section 6.2]{t05}.
Indeed, in this case, the infimum is attained in POCP (\ref{pocp}),
and Assumption \ref{norelaxationgap} is readily satisfied.

\begin{lemma}[Lower semi-continuity of the value of the measure LP]
The function $(t_0, x_0) \to p^*(t_0, x_0)$ is lower semicontinuous.
\end{lemma}

\begin{proof}
We need to show that given a sequence $(t_k, x_k)_{k\in \mathbb N}$ such that $\lim_{k \to \infty} (t_k, x_k) = (t,x) \in {\mathbb R}^{n+1}$, it holds that ${\liminf}_{k \to \infty} p^*(t_k, x_k) \geq p^*(t,x)$.
Suppose that $(t,x)$ is such that measure LP (\ref{measpocp}) is feasible. If the left-hand side is not finite, the result holds.
If the left-hand side is finite, we can consider, up to taking a subsequence, that
${\liminf}_{k \to \infty} p^*(t_k, x_k) = {\lim}_{k \to \infty} p^*(t_k, x_k) < \infty$.
Since the infimum is attained in measure LP (\ref{measpocp}), we have a sequence of measures $(\mu_k, \mu_{Tk})_{k \in \mathbb N}$ such that $p^*(t_k, x_k) = \left\langle\mu_k, l \right\rangle + \left\langle\mu_{Tk}, l_T \right\rangle$
and $\frac{\partial \mu_k}{\partial t}+\mathrm{div}\:f \mu_k + \mu_{Tk} = \delta_{t_k}\delta_{x_k}$.
Convergence of $(t_k, x_k)$ to $(t,x)$ implies weak-star convergence of $\delta_{t_k}\delta_{x_k}$ to $\delta_t\delta_x$.
Using the same closedness argument as in the proof of Lemma \ref{nodualitygap}, we can consider that, up to a subsequencce, $\mu_k$ and $\mu_{Tk}$ converge to  some measures $\mu$ and $\mu_T$ in the weak-star topology and that 
$\frac{\partial \mu}{\partial t}+\mathrm{div}\:f \mu + \mu_{T} = \delta_t\delta_x$. Hence, we have
${\liminf}_{k \to \infty} p^*(t_k, x_k) = \left\langle\mu, l \right\rangle + \left\langle\mu_{T}, l_T \right\rangle$
and the pair $(\mu, \mu_{T})$ is feasible for problem $p^*(t,x)$. Therefore
${\liminf}_{k \to \infty} p^*(t_k, x_k) \geq p^*(t, x)$
which proves the result when LP (\ref{measpocp}) is feasible for $(t,x)$. Using similar arguments, one can show that if $(t,x)$ is such that LP (\ref{measpocp}) is not feasible, there cannot be infinitely many $k$ such that LP (\ref{measpocp}) is feasible for $(t_k, x_k)$.
\end{proof}

The following result extends the convergence properties of the maximizing sequence. 
\begin{theorem}[Uniform convergence along relaxed trajectories]
\label{convergencerelaxed}
For any sequence $(v_k)_{k\in \mathbb N}$ admissible for the dual LP (\ref{contpocp}),  for any solution $(x, \omega_t)$
of relaxed POCP (\ref{relaxedpocp}), and for any $t \in [t_0,T]$, it holds
\[0 \leq v^*(t, x(t)) - v_k(t, x(t)) \leq v^*(t_0, x_0) - v_k(t_0, x_0) \underset{k \to \infty}{\rightarrow} 0.\]
\label{convergenceAccumulation}
\end{theorem}

\begin{proof}
Let $(x_j, u_j)_{j \in \mathbb N}$ be an approximating sequence for $(x, \omega_t)$, whose existence is guaranteed by Assumption \ref{norelaxationgap}. For any $j \in \mathbb N$, $k \in \mathbb N$, and $t \in [t_0, T]$, we have
$ l_T(x_j(T)) - v_k(T, x_j(T)) + \int_t^T (l(x_j(s), u_j(s)) + {\mathcal L}(s,x_j(s),u_j(s)))ds
= l_T(x_j(T)) + \int_t^T l(x_j(s), u_j(s))ds - v_k(t,x_j(t))$.
Both the first term and the integrand are positive in the left-hand side. Therefore, the right-hand side is a decreasing function of $t$. Moreover, the trajectory is suboptimal, and $v_k$ is a lower bound on the value function. It holds that
$0 \leq v^*(t, x_j(t)) - v_k(t,x_j(t)) \leq l_T(x_j(T)) + \int_t^T l(x_j(s), u_j(s))ds - v_k(t,x_j(t)) 
\leq l_T(x_j(T)) + \int_{t_0}^T l(x_j(s), u_j(s))ds - v_k(t_0,x_0)$.
Letting $j$ tend to infinity, using the lower semicontinuity of $v^*$, we conclude that
$0 \leq v^*(t, x(t)) - v_k(t,x(t)) \leq \liminf_{j \to \infty} v^*(t, x_j(t)) - v_k(t,x_j(t))
\leq \lim_{j \to \infty} ( l_T(x_j(T)) + \int_{t_0}^T l(x_j(s), u_j(s))ds - v_k(t_0,x_0)) = v^*(t_0, x_0) -  v_k(t_0,x_0)$.
\end{proof}
	
It is important to notice that Theorem \ref{convergencerelaxed} holds for any trajectory realizing the minimum of POCP (\ref{relaxedpocp}) and therefore, for all of them simultaneously. In addition, these trajectories are identified with limiting trajectories of POCP (\ref{pocp}) by Assumption \ref{norelaxationgap}.

\section{Optimal control over a set of initial conditions}\label{sec:superposition}

Liouville equation (\ref{liouville}) is used as a linear equality constraint in POCP (\ref{measpocp}) with a Dirac right-hand side
as an initial condition. However, this right-hand side can be replaced by more general probability measures. The linearity of the constraint allows to extend most of the results of the previous section to this setting. It leads to similar convergence guarantees regarding a (possibly uncountable) set of optimal control problems. These guarantees  hold for solutions of a single infinite-dimensional LP. 

Suppose that we are given a set of initial conditions $X_0 \subset X$, such that $(t_0,x_0) \in \mathscr{A}$
for every $x_0 \in X_0$. Given a probability measure $\xi_0 \in {\mathscr P}(X_0)$,  let
\[
\mu_0(dt,dx) = \delta_{t_0}(dt)\xi_0(dx) 
\]
and consider the following average value
\begin{equation}\label{averagevalue}
\bar{v}^*(\mu_0) := \int_{X_0} v^*(t, x)\mu_0(dt,dx) = \langle v^*,\mu_0 \rangle
\end{equation}
where $v^*$ is the value of POCP (\ref{pocp}). Under Assumption \ref{norelaxationgap}, by linearity
this value is equal to the value of POCP (\ref{measpocp}) with $\mu_0$ as the right-hand side of the equality constraint,
namely the primal averaged LP
\begin{equation}\label{avmeaspocp}
\begin{array}{llll}
\bar{p}^*(\mu_0) & := & \min_{\mu, \mu_T} &\langle l,\mu \rangle + \langle l_T,\mu_T \rangle \\
&&  \mathrm{s.t.} & \frac{\partial \mu}{\partial t}+\mathrm{div}\:f \mu + \mu_T = \mu_0 \\ 
&&& \mu \in {\mathscr M}_+([t_0,T]\times X\times U)\\
&&&  \mu_T \in {\mathscr M}_+(\{T\}\times X_T)
\end{array}
\end{equation}
with dual averaged LP
\begin{equation}\label{avcontpocp}
\begin{array}{llll}
\bar{d}^*(\mu_0) & := & \sup_v & \langle v,\mu_0 \rangle \\
&&  \mathrm{s.t.} & l+\frac{\partial v}{\partial t}+\mathrm{grad}\:v\cdot f \in {\mathscr  C}_+([t_0,T]\times X\times U) \\
&&& l_T-v(T,.) \in {\mathscr C}_+(X_T).
\end{array}
\end{equation}
The absence of duality gap is justified in the same way as in Lemma \ref{nodualitygap}. Moreover,  Lemma \ref{lowerbound} also holds, and, as in Lemma \ref{pointwiseConvergence},  we have the existence of maximizing lower bounds $v_k$ such that 
\[
\lim_{k\to\infty} \langle v_k,\mu_0 \rangle = \bar{v}^*(\mu_0) = \bar{p}^*(\mu_0) = \bar{d}^*(\mu_0).
\]
Intuitively, primal LP (\ref{avmeaspocp}) models a superposition of optimal control problems. The LP formulation allows to express it as a single program over measures satisfying a transport equation. A relevant question here is the relation between solutions of averaged measure LP (\ref{avmeaspocp}) and optimal trajectories of the original problem POCP (\ref{pocp}). The intuition is that measure solutions of LP (\ref{avmeaspocp}) represent a superposition of optimal trajectories of the relaxed POCP (\ref{relaxedpocp}). These trajectories are themselves limiting trajectories of the original  POCP (\ref{pocp}). The superposition principle of \cite[Theorem 3.2]{a08} allows to formalize this intuition and to extend the result of Theorem \ref{convergencerelaxed} to this setting.

\begin{theorem}[Uniform convergence on support of optimal measure]
\label{convergenceSupport}
For any solution $(\mu, \mu_T)$ of primal averaged LP (\ref{avmeaspocp}), there are parametrized measures
$\xi_t \in {\mathscr P}(X)$ (for the state) and $\omega_t \in {\mathscr P}(U)$ (for the control) such that
$\mu(dt,dx,du) = dt\:\xi_t(dx)\omega_t(du)$, $\mu_0(dt,dx)=\delta_{t_0}(dt)\xi_{t_0}(dx) $ and
$\mu_T(dt,dx)=\delta_{T}(dt)\xi_{T}(dx)$.
In addition, if $(v_k)_{k\in \mathbb N}$ is a maximizing sequence for dual averaged LP (\ref{avcontpocp}), for any $t \in [t_0,T]$,
it holds
\[0 \leq \int_X (v^*(t, x) - v_k(t, x)) \xi_t(dx)\leq \int_X (v^*(t_0, x_0) - v_k(t_0, x_0))\xi_0(dx_0) \underset{k \to \infty}{\rightarrow} 0.\]
\end{theorem}

\begin{proof}
		The decomposition is given by Lemma 3 in \cite{hk14}. It asserts the existence of a measure $\sigma \in {\mathscr M}({\mathscr C}([t_0, T], X))$ supported on trajectories admissible for relaxed POCP (\ref{relaxedpocp})
and such that for any measurable function $w: X \to \mathbb R$, it holds
$\int_X w(x) \xi_t(dx) = \int_{{\mathscr C}([t_0, T], X)} w(x(t)) \sigma(dx(.))$.
		By Assumption \ref{norelaxationgap}, all trajectories of the support of $\sigma$ are pointwise limits of sequences of feasible trajectories of POCP (\ref{pocp}). Hence $\sigma$-almost all of these sequences must be minimizing sequences for POCP (\ref{pocp}), otherwise, that would contradict optimality of $(\mu, \mu_T)$. The result follows by discarding the trajectories which are not limits of minimizing sequences. This does not change $\sigma$ or $\xi_t$. Theorem \ref{convergenceAccumulation} applies to $\sigma$-almost all these trajectories and we have
$0 \leq  \int_{{\mathscr C}([t_0, T], X)}(v^*(t, x(t)) - v_k(t,x(t)))\sigma(dx(.)) =
\int_X (v^*(t, x) - v_k(t,x))\xi_t(dx) \leq \int_X  (v^*(t_0, x_0) - v_k(t_0, x_0)) \xi_0(dx_0)$.
	\end{proof}

A remarkable practical implication of this result is that maximizing sequences of averaged dual LP (\ref{avcontpocp}) provide an approximation to the value function of POCP (\ref{pocp}) that is uniform in time and almost uniform in space along limits of optimal trajectories starting from $X_0$.

\section{Numerical illustrations}

In Sections \ref{sec:occupation} and \ref{sec:superposition} we reformulated nonlinear optimal control problems
as abstract linear conic optimization problems that involve manipulations of measures and continuous functions in their full generality. The results presented in Section \ref{sec:approx} are related to properties of minimizing or maximizing elements, or sequences of elements for these problems. From a practical point of view, it is possible to construct these sequences using the same numerical tools as in static polynomial optimization. On the primal side, this allows to approximate the minimizing elements of measure LP problems with a converging hierarchy of moment SDP problems. On the dual side, we can construct numerically maximizing sequences of polynomial SOS certificates for the continuous function LP problems. The convergence properties that we investigated hold in particular for these
solutions of the moment-SOS hierarchy.

This section illustrates convergence properties of the sequence of approximations of value functions computed using
moment-SOS hierarchies. We consider simple, but largely spread, optimal control problems for which the value function (or optimal trajectories) are known.

\subsection{Uniform approximation along an optimal trajectory}

Consider the one-dimensional turnpike POCP analyzed in Section 22.2 of \cite{c13}:

\begin{equation}\label{turnpike}
\begin{array}{llll}
v^*(t_0,x_0) := & \inf_{u} & \displaystyle \int_{t_0}^2 (x(t) + u(t))dt \\
& \mathrm{s.t.} & \dot{x}(t) = 1+x(t)- x(t) u(t)\\
&&    x(t_0) = x_0 \\
&& u(t) \in [0, 3].
\end{array}
\end{equation}

\begin{figure}[h!]
    \centering
\includegraphics[width=.5\textwidth]{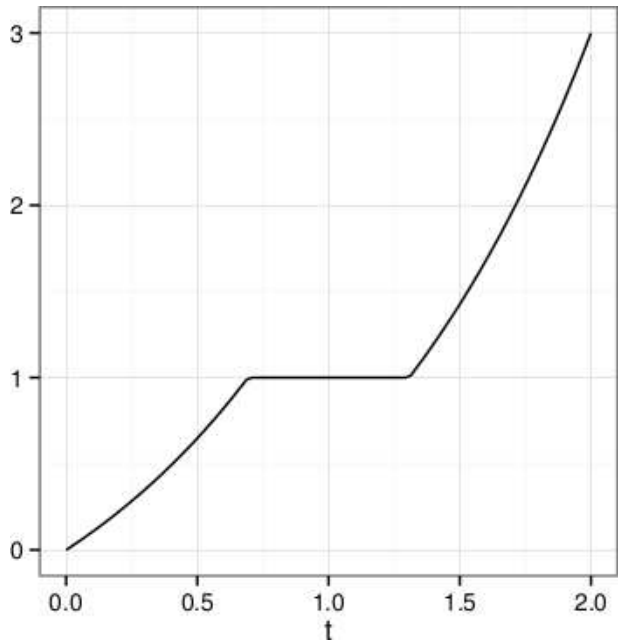}
    \caption{Optimal trajectory $t \mapsto x^*(t)$ starting at $(t_0,x_0)=(0,0)$ for the turnpike POCP (\ref{turnpike}).}
    \label{fig:trajturnpike}
\end{figure}
\begin{figure}[h!]
    \centering
\includegraphics[width=.5\textwidth]{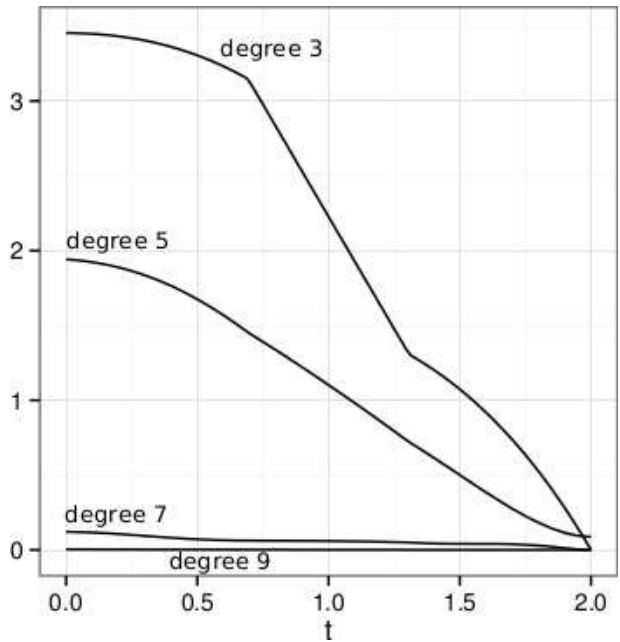}
    \caption{Differences $t \mapsto v^*(t,x^*(t))-v_k(t,x^*(t))$ between the actual value function
  and its polynomial approximations of increasing degrees $k=3,5,7,9$
  along the optimal trajectory $t \mapsto x^*(t)$ starting at $(t_0,x_0)=(0,0)$ for the turnpike POCP (\ref{turnpike}).
We observe uniform convergence along this trajectory, as well as time decrease of the difference,
as predicted by the theory.}
    \label{fig:valueturnpike}
\end{figure}

For this problem, the infimum is attained at a unique optimal control which is piecewise constant.
The optimal trajectory $t \mapsto x^*(t)$ starting at $(t_0,x_0) = (0,0)$ is presented in Figure \ref{fig:trajturnpike}.
The uniform convergence of approximate value functions $t \mapsto v_k(t,x^*(t))$ to the true value function 
$t \mapsto v^*(t,x^*(t))$ along this optimal trajectory,  stated by Theorem \ref{convergenceAccumulation},
is illustrated in Figure \ref{fig:valueturnpike}. Moreover, the difference $t \mapsto v^*(t,x^*(t))-v_k(t,x^*(t))$
is a decreasing function of time as we observed in the proof of Theorem \ref{convergenceAccumulation}.

\subsection{Uniform approximation over a set}

Consider the classical linear quadratic regulator problem:
\begin{equation}\label{lqr}
\begin{array}{llll}
v^*(t_0,x_0) :=&    \inf_{u} & \displaystyle \int_{0}^1 (10 x(t)^2 + u(t)^2)dt \\
&  \mathrm{s.t.} & \dot{x}(t) = x(t) + u(t)\\
&&    x(t_0) = x_0.
\end{array}
\end{equation}

\begin{figure}[h!]
    \centering
    \includegraphics[width=.5\textwidth]{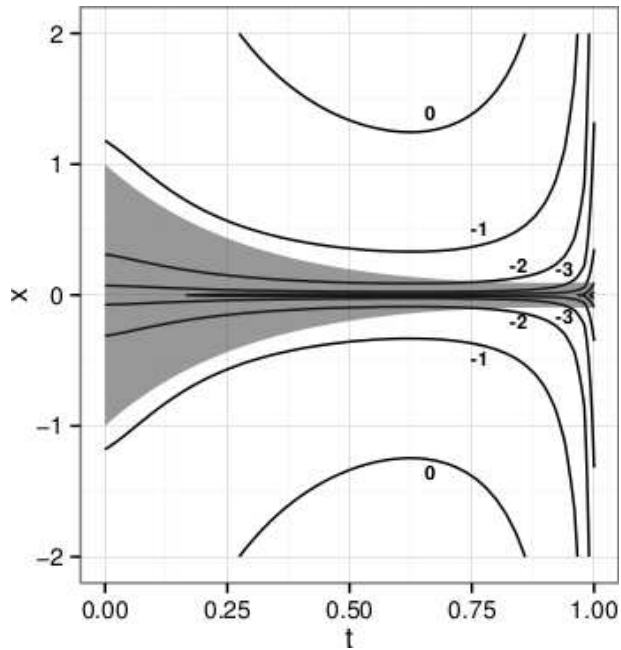}
    \caption{Countour lines (at $0,-1,-2,-3,\ldots$) of the decimal logarithm of the difference $(t,x) \mapsto v^*(t,x)-v_6(t,x)$
between the actual value function and its polynomial approximation of degree 6 for LQR POCP (\ref{lqr}).
The dark area represents the set of optimal trajectories starting from $x_0 \in X_0=[-1,1]$ at time $t_0=0$.
We observe that the difference is smaller in this area, as predicted by the theory.\label{fig:lqr}}.
\end{figure} 

For each $(t_0,x_0)$, the infimum is attained and the value of the problem can be computed
by solving a Riccati differential equation. To illustrate Theorem \ref{convergenceSupport}
we are interested  in the average value (\ref{averagevalue}) for an initial
measure $\mu_0$ concentrated at time $0$ and uniformly distributed in space
in $X_0 = [-1, 1]$. We approximate this value with primal and dual solutions of LPs (\ref{avmeaspocp}) and (\ref{avcontpocp}).
The countour lines (in decimal logarithmic scale) of the difference between the true value function $(t,x) \mapsto v^*(t,x)$
and a polynomial approximation of degree 6 $(t,x) \mapsto v_6(t,x)$ is represented in Figure \ref{fig:lqr}.
We also show the support of optimal trajectories starting from $X_0$. This illustrates the fact that
the approximation of the value function is correct in this region, as stated by Theorem \ref{convergenceSupport}.
It is noticeable that this is computed
by a single linear program and provides approximation guarantees uniformly over an uncountable set of
optimal control problems.

\end{document}